\documentclass[10pt,a4paper]{amsart}
\usepackage{amsmath,amssymb,amscd}
\usepackage{tikz}

\newtheorem{theorem}{Theorem}[section]
\newtheorem{corollary}[theorem]{Corollary}
\newtheorem{lemma}[theorem]{Lemma}
{\theoremstyle{definition}
\newtheorem{definition}{Definition}[section]
 
 }
 {\theoremstyle{remark}

}

\newcommand{\fs}{{W}}
\newcommand{\cyl}{\mathfrak C}
\newcommand{\bl}{I}

\newcommand{\dist}{\operatorname{dist}}

\newcommand{\Id}{\operatorname{Id}}

\newcommand{\vect}{\operatorname{span}}
\newcommand{\Hol}{\operatorname{Hol}}
\newcommand{\Lip}{\operatorname{Lip}}
\newcommand{\var}{{\text{var}}}

\renewcommand{\j}{\mathfrak i}
\def\bB{\hat{B}}
\def\<{\left\langle}
\def\>{\right\rangle}
\def\KR{{\text{\textsc{KR}}}}
\def\dif{\text{ d}} 
\def\R{\mathbf R}

\def\car{\mathbf 1}

\def\V{V}
\newcommand{\esp}[1]{\mathbf E\left[#1\right]}

\def\wB{B}

\def\ws{W}
\def\cm{{I_{1,2}}}
\def\fs{{W}}

\def\var{{\text{var}}}
\def\c1var{{C^{1-\text{var}}}}

\def\cA{{\mathcal A}}

\def\Id{\mbox{Id}}
\renewcommand{\j}{\mathfrak i}

\def\dist{\operatorname{dist}}

\def\N{{\mathbf N}}
\def\fs{{W}}

\def\esp#1{\mathbf E\left[#1\right]}

\def\dif{\text{ d}}
\def\R{\mathbf{R}}
\def\car{\mathbf{1}}

\def\KR{{\text{\textsc{kr}}}}

\def\Hol{\mbox{Hol}}
\def\KR{{\text{\tiny{\scshape{KR}}}}}
\def\Lip{\operatorname{Lip}}
\def\Hol{\mbox{Hol}}

\def\dyad{{\mathcal D}}
\def\V{{\mathcal V}}

\def\div{\delta_{\eta,p}}

\def\<{\left \langle}
  \def\>{\right\rangle}

\def\tS{S}
\def\emb{{\mathfrak e}}
\def\j{\emb}
\def\dyad{\mathcal D}
\def\A{{\mathcal A}}
\def\var#1{\operatorname{Var}\left[#1\right]}
\def\/{\, |\, }

\begin{document}
\title{Donsker's theorem in  Wasserstein-1 distance}
\author{L. Coutin}
\address{Institute  of Mathematics\\
  Universit\'e Toulouse 3\\
  Toulouse, France} \email{laure.coutin@math.univ-toulouse.fr}
\author{L. Decreusefond}
\address{LTCI, T\'l\'ecom Paris, Institut polytechnique de  Paris\\
  Paris, France} \email{Laurent.Decreusefond@mines-telecom.fr}

\thanks{} \thanks{The first  author is partially supported by ANR MESA}
\begin{abstract}
We compute the  Wassertein-1 (or Kolmogorov-Rubinstein) distance between a random
walk in $\R^{d}$ and the Brownian motion. The proof is based on a new estimate
of the Lipschitz modulus  of the solution of the Stein's equation. As an
application, we can evaluate the rate of convergence towards the local time at
$0$ of the Brownian motion.
\end{abstract}
\maketitle{}

\section{Motivations}
\label{sec:motiv}
For a complete, separable metric space $X$,  the
topology of convergence in distribution is metrizable
\cite{DudleyRealanalysisprobability2002} by  considering the
so-called Kolmogorov-Rubinstein or Wasserstein-1 distance:
\begin{equation}\label{eq_core:1}
  \dist_{\KR}(\mu,\nu)=\sup_{F\in \Lip_{1}(X)} \left( \int_{X} F\dif \mu -\int_{X}F\dif \nu \right)
\end{equation}
where
\begin{equation*}
  \Lip_{1}(X)=\left\{ F\, :\, X\to \R, |F(x)-F(y)|\le \dist_{X}(x,y),\ \forall x,y\in X \right\}.
\end{equation*}
The formulation \eqref{eq_core:1} is well
suited to evaluate distance by the Stein's method. When $X=\R$, there is no
particular difficulty to evaluate the K-R distance when $\mu$ is the Gaussian
distribution. When, $X=\R^{d}$, it is only recently (see
\cite{Fang2018,Gallouet2018,Raic2018} and references therein) that some improvement of the
standard Stein's method has been proposed to get the K-R distance to the
Gaussian measure on $\R^{d}$. The bottleneck is the estimate of the Lipschitz modulus  of the second order derivative of the solution of the Stein's
equation when $F$ is only assumed to be Lipschitz continuous. Namely, for $f\, :\, \R^{d}\to \R$, for any $t>0$, consider the function
\begin{equation*}
P_{t}f\, :\,   x\in\R^{d}\longmapsto  \int_{\R^{d}}f(e^{-t}x+\sqrt{1-e^{-2t}}y)\dif \mu_{d}(y)
\end{equation*}
where $\mu_{d}$ is the standard Gaussian measure on $\R^{d}$.  In dimension $1$, the Stein's equation reads as
  \begin{equation*}
    -xh(x)+h'(x)=f(x)-\int_{\R}f\dif \mu_{1},
  \end{equation*}
  so that
  \begin{equation}\label{eq_core:16}
    h(x)=\int_{0}^{\infty }P_{t}f(x)\dif t
  \end{equation}
  and the subsequent computations require to evaluate only the Lipschitz modulus  of $h'$. 
For $f\in L^{1}(\mu)$,
it is classical to see that $P_{t}f$ is infinitely  differentiable and that
\begin{equation}
  \label{eq_core:2}
  (P_{t}f)^{(k)}(x)=\left( \frac{e^{-t}}{\sqrt{1-e^{-2t}}} \right)^{k}\int_{\R^{d}}f(e^{-t}x+\sqrt{1-e^{-2t}}y)H_{k}(y)\dif \mu_{d}(y)
  \end{equation}
  where $H_{k}$ is the $k$-th Hermite polynomial. On the other hand, if $f$ is
  $k$-times differentiable, we have
  \begin{equation}\label{eq_core:3}
     (P_{t}f)^{(k)}=e^{-kt}P_{t}(f^{(k)}).
  \end{equation}
 According to
  \eqref{eq_core:2}, we get
  \begin{equation*}
    h'(x)=\int_{0}^{\infty } \frac{e^{-t}}{\sqrt{1-e^{-2t}}} \int_{\R^{d}}f(e^{-t}x+\sqrt{1-e^{-2t}}y)y\dif \mu_{1}(y)\dif t.
  \end{equation*}
  It is apparent that the Lipschitz  modulus  of $h'$ simply depends on the
  Lipschitz modulus  of $f$. However, in higher dimension, the Stein's equation
  becomes
   \begin{equation}\label{eq_core:17}
    -x.\nabla h(x)+\Delta h(x)=f(x)-\int_{\R^{d}}f\dif \mu_{d},
  \end{equation}
  whose solution is formally given by \eqref{eq_core:16}.   
The form of \eqref{eq_core:17} entails that we need to estimate the Lipschitz modulus  of $\Delta h$, which requires to
  use \eqref{eq_core:2} for $k=2$. Unfortunately, we have to realize that 
  \begin{equation*}
    \left( \frac{e^{-t}}{\sqrt{1-e^{-2t}}} \right)^{k} \notin L^{1}([0,+\infty);\dif t).
  \end{equation*}
  Hence, until the very recent papers \cite{Fang2018,Raic2018}, the strategy was  to assume that
  $\nabla f$ is Lipschitz, apply once \eqref{eq_core:3} to compute the first derivative
  of $P_{t}f$ and then apply \eqref{eq_core:2} to this expression:
  \begin{equation*}
    \Delta h(x)=\int_{0}^{\infty } \frac{e^{-t}}{\sqrt{1-e^{-2t}}} \int_{\R^{d}}\nabla f(e^{-t}x+\sqrt{1-e^{-2t}}y).y\dif \mu_{d}(y)\dif t.
  \end{equation*}
This means that instead of computing the supremum in the right-hand-side of \eqref{eq_core:1}, over Lipschitz functions, it
is computed over functions whose first derivative is Lipschitz. This also
defines a distance, which  does not
change the induced topology but the accuracy of the bound is degraded. 

  In infinite dimension, a new problem arises which is best explained by going
  back to the roots of the Stein's method in dimension~$1$. Consider that we
  want to estimate the K-R distance in the standard Central Limit Theorem. Let
  $(X_{n},\, n\ge 1)$ be a sequence of independent, identically distributed
  random variables with $\esp{X}=0$ and $\esp{X^{2}}=1$. Let
  $T_{n}=n^{-1/2}\sum_{j=1}^{n}X_{j}$. The Stein-Dirichlet representation
  formula \cite{DecreusefondSteinDirichletMalliavinmethod2015} states that
  \begin{equation}
    \label{eq:SteinDirichlet}
    \esp{f(T_{n})}-\int_{\R}f\dif \mu_{1}=\esp{\int_{0}^{\infty} LP_{t}f(T_{n})\dif t}
  \end{equation}
  where
  \begin{equation*}
    Lf(x)=-xf(x)+f'(x)=L_{1}f(x)+L_{2}f(x),
  \end{equation*}
  with obvious notations. Now,
  \begin{equation*}
    L_{1}P_{t}f(T_{n})=-T_{n}(P_{t}f)'(T_{n})=-\frac{1}{\sqrt{n}}\sum_{j=1}^{n} X_{j}(P_{t}f)'(T_{n}).
  \end{equation*}
The trick, which amounts to an integration by parts for a Malliavin structure on
independent random variables (see \cite{Decreusefond_2018}), is to write
\begin{equation*}
  \esp{X_{j}(P_{t}f)'(T_{n})}= \esp{X_{j}\Bigl( (P_{t}f)'(T_{n})-(P_{t}f)'(T_{n}-X_{j}/\sqrt{n}) \Bigr)}
\end{equation*}
in view of the independence of the random variables. Then, we use the
fundamental theorem  of calculus
 in this expression around the point $T_{n}^{\neg
  j}=T_{n}-X_{j}/\sqrt{n}$:  
\begin{multline*}
  \esp{X_{j}\Bigl( (P_{t}f)'(T_{n})-(P_{t}f)'(T_{n}-X_{j}/\sqrt{n}) \Bigr)}\\=\frac{1}{\sqrt{n}}\int_{0}^{1} \esp{X_{j}^{2} (P_{t}f)''(T_{n}+r X_{j}/\sqrt{n}) \Bigr) }\dif r.
\end{multline*}
Since,
\begin{equation*}
  \int_{0}^{1} \esp{X_{j}^{2} (P_{t}f)''(T_{n}^{\neg
  j})}\dif r=\esp{(P_{t}f)''(T_{n}^{\neg
  j})},
\end{equation*}
we get
\begin{multline}\label{eq_core:4}
  LP_{t}f(T_{n})\\=-\frac{1}{n}\sum_{j=1}^{n } \int_{0}^{1} \esp{X_{j}^{2}
    \Bigr((P_{t}f)''(T_{n}^{\neg
  j}+r X_{j}/\sqrt{n})-(P_{t}f)''(T_{n}^{\neg
  j})
    \Bigr) }\dif r\\
  +\frac{1}{n}\sum_{j=1}^{n } \esp{(P_{t}f)''(T_{n}^{\neg
  j}) -(P_{t}f)''(T_{n})}.
\end{multline}
This formula confirms that the crux of the matter is now to estimate
uniformly the Lipschitz  modulus  of $(P_{t}f)''$. It also shows how we get the
order of convergence. We have one occurrence of $n^{-1/2}$ in the definition of
$T_{n}$, which appears in the expression of $L_{1}$. The same factor appears a
second time when we proceed to the Taylor expansion and then, it will appear a
third time when we plug~\eqref{eq_core:2} into~\eqref{eq_core:4}. This means
that we have a factor $n^{-3/2}$ which is summed up $n$ times, hence the rate of
convergence which is known to be  $n^{-1/2}$.

Now, if we are interested in the Donsker theorem, the process whose limit we
would like to assess is
\begin{equation*}
  \tS_{n}(t)=\sum_{j=1}^{n} X_{j} h_{j}^{n}(t)
\end{equation*}
where
\begin{equation*}
  h_{j}^{n}(t)=\sqrt{n}\int_{0}^{t} \car_{[j/n,(j+1)/n)}(s)\dif s.
\end{equation*}
For reasons that will be explained below, the analog of the second order
derivatives will involve
\begin{equation}\label{eq_core:5}
  \<
    h_{j}^{\otimes 2},   \nabla^{(2)}(P_{t}f)(S_{n}^{\neg
  j}S_{n}+r X_{j}/\sqrt{n})-\nabla^{(2)}(P_{t}f)(S_{n}^{\neg
  j})
  \>_{I_{1,2}^{\otimes 2}}
\end{equation}
where $\nabla$ is the Malliavin derivative,  $I_{1,2}$ is the Cameron-Martin space
\begin{equation*}
  I_{1,2}=\left\{f, \exists ! \dot f \in L^{2}([0,1],\dif t) \text{ with } f(t)=\int_{0}^{t}\dot f(s)\dif s\right\}
\end{equation*}
and
\begin{equation*}
  \|f\|_{I_{1,2}}=\|\dot f\|_{L^{2}}.
\end{equation*} Recall that in the context of Malliavin calculus, this space is
identified to its dual which means that the dual of $L^{2}$ is not itself.
The difficulty is then that we do not have a $n^{-1/2}$ factor in the definition
of $\tS_{n}$ and it is easily seen that $\|h_{j}^{n}\|_{I_{1,2}}=1$, hence no
multiplicative factor will pop up in~\eqref{eq_core:5}. In
\cite{coutin_convergence_2017}, we bypassed this difficulty by assuming enough
regularity of $f$ so that $\nabla^{(2)}P_{t}f$ belong to the dual of $L^{2}$.
Then, in the estimate of terms as those appearing in \eqref{eq_core:5}, it is  the
$L^{2}$-norm of $h_{j}^{n}$ which  appears and it turns out that
$\|h_{j}^{n}\|_{L^{2}}\le c\,n^{-1/2}$, hence the presence of a factor $n^{-1}$, which
saves the proof.

The goal of this paper is to weaken the hypothesis on $f$ to be able to
upper-bound the true K-R distance between the distribution of $\tS_{n}$ and
the distribution of a Brownian motion, that is
\begin{equation*}
  \sup_{f\in \Lip_{1}(X)}\esp{f(\tS_{n})}-\esp{f(B)}.
\end{equation*}
The space $X$ is a Banach space we can choose arbitrarily as far as it
can be equipped with the structure of an  abstract Wiener space and it
contains the sample paths of $\tS_{n}$ and $B$.

The main technical  result of this article is Theorem~\ref{thm:majoModulusOfContinuity} which
gives a new estimate of the Lipschitz modulus  of $\nabla^{(2)}P_{t}f$ for
$t>0$. The main idea is to introduce a hierarchy  of approximations. There is a
first scale induced by the time discretization coming from the definition of
$\tS_{n}$. Then,  we consider a coarser  discretization onto which we project our
approximations in order to benefit from the averaging effect of the ordinary
CLT. It turns out that the optimal ratio is obtained when the mesh of the
coarser subdivision is roughly the cubic root of the mesh of the reference
partition. Moreover, after \cite{coutin_steins_2013} and
\cite{coutin_convergence_2017}, we are convinced that it is simpler and as
efficient to stick to finite dimension as long as possible. For, we consider the
affine interpolation of the Brownian motion as an intermediary process. The  distance
between the Brownian sample-paths and their affine interpolation is well known.
This reduces the problem to estimate the distance between $\tS_{n}$ and the affine
interpolation of~$B$, a task which can be handled by the Stein's method. It turns out
that the bottleneck is in fact the rate of convergence of the Brownian
interpolation to the Brownian motion.

This paper is organized as follows. In Section~\ref{sec:prelim}, we show how to
view  fractional Sobolev spaces as  Wiener spaces. In
Section~\ref{sec:donsk-theor-ws_eta}, we explain the line of thoughts we used.
The proofs are given in Section~\ref{sec:proofs}.

\section{Preliminaries}
\label{sec:prelim}

\subsection{Fractional Sobolev spaces}

\label{sec:fract-sobol-spac}
As in \cite{decreusefond_stochastic_2005,friz_multidimensional_2010}, we consider the fractional
Sobolev spaces $\fs_{\eta,p}$ defined for $\eta \in (0,1)$ and $p\ge 1$
as the the closure of ${\mathcal
  C}^1$ functions with respect to the norm
\begin{equation*}
  \|f\|_{\eta,p}^p=\int_0^1 |f(t)|^p \dif t + \iint_{[0,1]^2}
  \frac{|f(t)-f(s)|^p}{|t-s|^{1+p\eta}}\dif t\dif s.
\end{equation*}
For $\eta=1$, $\fs_{1,p}$ is  the completion of $\mathcal C^1$ for
the norm:
\begin{equation*}
   \|f\|_{1,p}^p=\int_0^1 |f(t)|^p \dif t + \int_0^1 |f^\prime(t)|^p \dif t.
\end{equation*}
They are known to be Banach spaces and to satisfy the Sobolev embeddings \cite{adams_sobolev_2003,feyel_fractional_1999}:
\begin{equation*}
  \fs_{\eta,p}\subset \Hol(\eta-1/p) \text{ for } \eta-1/p>0
\end{equation*}
and
\begin{equation*}
   \fs_{\eta,p}\subset  \fs_{\gamma,q}\text{ for } 1\ge \eta\ge \gamma  \text{
     and } \eta-1/p\ge \gamma-1/q.
 \end{equation*}
 As a consequence, since $\fs_{1,p}$ is separable (see
 \cite{brezis_analyse_1987}), so does  $\fs_{\eta,p}$. We need to compute the
 $\fs_{\eta,p}$ norm of primitive of step functions.
\begin{lemma}
 \label{lem:normeHDansDual}
 Let $0\le s_{1} < s_{2}\le 1$ and consider
 \begin{equation*}
   h_{s_{1},s_{2}}(t)=\int_{0}^{t} \car_{[s_{1},s_{2}]}(r)\dif r.
 \end{equation*}
 There exists $c>0$ such that for any $s_{1},s_{2}$, we
  have
  \begin{equation}
    \label{eq_donsker:3}
     \|h_{s_{1},s_{2}}\|_{{\eta,p}} \le c\, |s_{2}-s_{1}|^{1/2-\eta}.
  \end{equation}
\end{lemma}
\begin{proof}
     Remark that for any $s,t\in [0,1]$,
  \begin{equation*}
    \left| h_{s_{1},s_{2}} (t)-h_{s_{1},s_{2}}(s) \right|\le |t-s|\wedge (s_{2}-s_{1}).
  \end{equation*}
  The result then follows from the definition of the $\fs_{\eta,p}$
    norm.
\end{proof}
We denote by $\ws_{0,\infty}$ the space of continuous (hence bounded) functions
on $[0,1]$ equipped with the uniform norm.
\subsection{Fractional spaces $\ws_{\eta,p}$ as Wiener spaces}
\label{sec:gelfand-triplet-1}

 Let
  \begin{equation*}
    \Lambda=\{(\eta,p)\in \R^{+}\times\R^{+}, 0< \eta-1/p<1/2\}\cup\{(0,\infty)\}.
  \end{equation*}
In what follows, we always choose $\eta$ and $p$ in $\Lambda$. 
Consider $(Z_{n},\, n\ge 1)$ a sequence of independent, standard Gaussian random
variables and let $(z_{n},\, n\ge 1)$ be a complete orthonormal basis of
$I_{1,2}$. Then, we know from \cite{ito_convergence_1968} that 
\begin{equation}\label{eq_core:6}
  \sum_{n=1}^{N} Z_{n}\, z_{n} \xrightarrow{N\to \infty} B:=\sum_{n=1}^{\infty} Z_{n}\, z_{n} \text{ in } W_{\eta,p} \text{ with probability } 1,
\end{equation}
where $B$ is a Brownian motion. We clearly have the diagram
\begin{equation}\label{eq_core:8}
  W_{\eta,p}^{*}\xrightarrow{\emb_{\eta,p}^{*}} (I_{1,2})^{*}\simeq I_{1,2}\xrightarrow{\emb_{\eta,p}}  W_{\eta,p},
\end{equation}
where $\emb_{\eta,p}$ is the embedding from $I_{1,2}$ into $W_{\eta,p}$. 
The space $I_{1,2}$ is dense in $W_{\eta,p}$ since polynomials do belong to
$I_{1,2}$. Moreover, Eqn.~\eqref{eq_core:6} and the Parseval identity entail
that for any $z\in \ws^{*}$,
  \begin{equation}\label{eq_core:7}
\esp{e^{i\< z,B\>_{\ws^{*}_{\eta,p},\ws_{\eta,p}}}}=\exp\left(-\frac12 \|\emb_{\eta,p}^* (z)\|_{I_{1,2}}^{2}\right).
\end{equation}
We denote by $\mu_{\eta,p}$ the law of $B$ on $\ws_{\eta,p}$. Then,
 the diagram \eqref{eq_core:8} and the identity~\eqref{eq_core:7} mean that
$(I_{1,2},\ws_{\eta,p},\mu_{\eta,p})$ is {a Wiener space}.
\begin{definition}[Wiener integral]
  The Wiener integral, denoted as $\delta_{\eta,p}$, is the isometric extension
  of the map
  \begin{align*}
   \delta_{\eta,p}\, :\,  \j_{\eta,p}^*(\fs_{\eta,p}^*)\subset\bl_{1,2}&\longrightarrow  L^2(\mu_{\eta,p})\\
   \j_{\eta,p}^*(\eta)  &\longmapsto \<\eta,\, y\>_{\fs_{\eta,p}^*,\fs_{\eta,p}}. 
  \end{align*}
\end{definition}
This means that if $h=\lim_{n\to \infty } \j_{\eta,p}^{*}(\eta_{n})$ in
$\bl_{1,2}$,
\begin{equation*}
  \delta_{\eta,p} h(y)=\lim_{n\to \infty} \<\eta_{n},\, y\>_{\fs_{\eta,p}^*,\fs_{\eta,p}} \text{ in }L^{2}(\mu_{\eta,p}).
\end{equation*}
\begin{definition}[Ornstein-Uhlenbeck semi-group]
  For any Lipschitz function on $\ws_{\eta,p}$, for any $\tau \ge 0$,
  \begin{equation*}
    P_\tau f(x)=\int_{\ws_{\eta,p}} f(e^{-\tau}x+\beta_\tau y)\dif\mu_{\eta,p}(y)
  \end{equation*}
where $\beta_\tau=\sqrt{1-e^{-2\tau}}$.
\end{definition}
The dominated convergence theorem entails that $P_\tau$ is ergodic:
For any $x\in \ws_{\eta,p}$, with probability~$1$, 
\begin{equation*}
  P_\tau f(x)\xrightarrow{\tau\to\infty} \int_{\ws_{\eta,p}} f\dif\mu_{\eta,p}.
\end{equation*}
Moreover, the invariance by rotation of Gaussian measures implies that
\begin{equation*}
  \int_{\ws_{\eta,p}} P_\tau f(x)\dif\mu_{\eta,p}(x)=
  \int_{\ws_{\eta,p}}f\dif\mu_{\eta,p}\text{, for any $\tau \ge 0$.}
\end{equation*}
Otherwise stated, the Gaussian measure on $\fs_{\eta,p}$ is the invariant and stationary
measure of the semi-group $P=(P_\tau, \, \tau \ge 0)$. For details on the
Malliavin gradient, we refer to \cite{Nualart1995b,Ustunel2010}.
\begin{definition}
  \label{def_donsker_final:2} Let $X$ be a Banach space. A function $f\, :\, \ws_{\eta,p}\to
  X$ is said to be cylindrical if it is of the form
  \begin{equation*}
    f(y)=\sum_{j=1}^k f_j(\delta_{\eta,p} h_1(y),\cdots,\delta_{\eta,p} h_k(y))\, x_j
  \end{equation*}
where for any $j\in \{1,\cdots,k\}$, $f_j$ belongs to the Schwartz space on $\R^k$, $(h_1,\cdots, h_k)$ are
elements of $\bl_{1,2}$ and $(x_1,\cdots,x_k)$ belong to $X$. The set of such
functions is denoted by $\cyl(X)$.

For $h\in \bl_{1,2}$,
\begin{equation*}
  \< \nabla f, \, h\>_{ \bl_{1,2}}=\sum_{j=1}^k\sum_{l=1}^{k} \partial_lf(\delta_{\eta,p} h_1(y),\cdots,\delta_{\eta,p} h_k(y))\, \<h_l,\, h\>_{ \bl_{1,2}}\ x_j,
\end{equation*}
which is equivalent to say
\begin{equation*}
  \nabla f = \sum_{j,l=1}^k \partial_jf(\delta_{\eta,p} h_1(y),\cdots,\delta_{\eta,p} h_k(y))\, h_l\otimes\ x_j.
\end{equation*}
 \end{definition}
 It is proved in \cite[Theorem 4.8]{shih_steins_2011} that
 \begin{theorem}
   For $f\in \Lip_{1}(\ws_{\eta,p})$, for any $t>0$, for any $x\in \ws_{\eta,p}$
   \begin{equation}
     \label{eq:dSurDtPtF}
     \frac{d}{dt}P_{t}f(x)=-\<x,\, \nabla(P_{t}f)(x)\>_{\ws_{\eta,p},\ws_{\eta,p}^{*}}+\sum_{i=1}^{\infty} \<\nabla^{(2)}P_{t}f(x),\, h_{i}\otimes h_{i}\>_{I_{1,2}}
   \end{equation}
   where $(h_{n}, \, n\ge 1)$ is complete orthonormal basis of $H$.
 \end{theorem}
 Note that a non trivial part of this theorem is to prove that the terms are
 meaningful: that $\nabla P_{t}f$ has values in $\ws_{\eta,p}^{*}$ instead of
 $I_{1,2}$ and that $\nabla^{(2)}P_{t}f(x)$ is trace-class.
 Actually, we only need a finite dimensional version of this identity in which
 all these difficulties do not appear. 

\section{Donsker's theorem in $\ws_{\eta,p}$}

\label{sec:donsk-theor-ws_eta}

For $m\ge 1$, let $\dyad^{m}=\{i/m,\, i=0,\cdots, m\}$, the regular
subdivision of the interval $[0,1]$. Let
\begin{equation*}
  \A^{m}=\{1,\cdots,d\}\times \{0,\cdots,m-1\}
\end{equation*}
and for $a=(a_{1},\, a_{2})\in \A^{m}$
\begin{equation*}
  h_{a}^{m}(t)=\sqrt{m}\,\int_{0}^{t}\car_{[a_{2},a_{2}+1/m)}(s)\dif s\  e_{a_{1}}.
\end{equation*}
Consider
\begin{equation*}
  \tS^{m}=\sum_{a\in \A^{m}} X_{a}\, h_{a}^{m}
\end{equation*}
where $(X_{a},\, a\in \A^{m})$ is a family of independent identically
distributed, $\R^{d}$-valued, random variables. We denote by $X$  a random
variable which has their common distribution. Moreover, we assume that
$\esp{X}=0$ and $\esp{\|X\|_{\R^{d}}^{2}}=1$. 
Remark that $(h_{a}^{m},a\in \A^{m})$ is an orthonormal family in $\R^{d}\otimes
I_{1,2}:=I_{1,2}^{d}$. Let
\begin{equation*}
  \V^{m}=\vect(h_{a}^{m},\,a\in \A^{m})\subset I_{1,2}^{d}
.
\end{equation*}
For any $m>0$, the map $\pi^{m}$ is the orthogonal projection from
$H:=I_{1,2}^{d}$
onto $\V^{m}$.
Let $0<N<m$, for $f\in \Lip_{1}(\ws_{\eta,p})$, we write
\begin{align} \label{dec-A}
 \esp{f(\tS^{m})} - \esp{f(B)} =\sum_{i=1}^3 A_i
\end{align}
where 
\begin{align*}
A_{1}&=\esp{f(\tS^{m})}-\esp{f( \pi^{N}(\tS^{m})}\\
A_2&= \esp{f\circ \pi^{N}(\tS^{m})}-\esp{f\circ \pi^N(B^{m})},\\
A_3&=\esp{f\circ \pi^N(B^{m})}-\esp{f(B)},
\end{align*}
where $B^{m}$ is the affine interpolation of the Brownian motion:
\begin{equation*}
  B^{m}(t)=\sum_{a\in\A^{m}}\sqrt{m}\ \bigl(B_{a_{1}}(a_{2}+\frac1{m})-B_{a_{1}}(a_{2})\bigr) \ h_{a}^{m}(t).
\end{equation*}
The two terms $A_{1}$ and $A_{3}$ are of the same nature: We have to compare two
processes which live on the same probability space. Since $f$ is Lipschitz, we
can proceed by comparison of their sample-paths. The term $A_{2}$ is different
as the two processes involved live on different probability spaces. This is for
this term that the Stein's method will be used.

We know from \cite{friz_multidimensional_2010} that
\begin{theorem}
  For any $(\eta,p)\in \Lambda,$ there exists $c>0$ such that 
  \begin{equation}\label{eq_core:10}
  \sup_{N} N^{1/2-\eta} \,\esp{\|B^{N}-B\|_{{\eta,p}}^{p}}^{1/p}\le c.
  \end{equation}
\end{theorem}
Moreover, we have
\begin{theorem}\label{propA1}
Let $(\eta,p)\in \Lambda.$ Assume that $X\in L^{p}(\ws;\R^{d},\mu_{\eta,p})$.
There exists a constant $c>0$ such that
\begin{equation*}
  \sup_{m,N}  N^{\frac{1}{2} -\eta}\, \esp{ \| \tS^{m} - \pi^N(\tS^{m}) \|_{{\eta,p}}^p}^{1/p} \le c \,\|X\|_{L^p}
  .
\end{equation*}
\end{theorem}
This upper-bound is far from being optimal and it is likely that it could be
improved to obtain a factor $N^{1-\eta}$. However, in view of
\eqref{eq_core:10}, it would bring no improvement to our final result.
\begin{theorem}\label{propA2} Let $(\eta,p)\in \Lambda.$
  Let $X_a$ belong to $L^p(\ws;\R^{d},\mu_{\eta,p})$ for some $p\ge 3$. Then,
  there exists $c>0$ such that for any $f \in \Lip_{1}(\ws_{\eta, p})$,
  \begin{equation}\label{eq_donsker_wiener:proj-bis}
\esp{f(\pi^N (\tS^m)) }-\esp{ f(\pi^N(B^m))} \le c\,\|X\|_{L^{p}}\ \frac{N^{1+\eta}}{\sqrt{m}}\ln( \frac{N^{1+\eta}}{\sqrt{m}})\cdotp
  \end{equation}
\end{theorem}
The global upper-bound for \eqref{dec-A} is proportional to 
\begin{equation*}
  N^{-1/2+\eta}+N^{1+\eta}m^{-1/2}\ln(N^{1+\eta}m^{-1/2}).
\end{equation*}
See $N$ as a function of $m$ and note that this expression is minimal for $N\sim m^{1/3}$.
Plug this into the previous expressions to obtain the main result of this paper:
\begin{theorem}
  \label{thm:main}
  Assume that $X\in L^{p}(\ws;\R^{d},\mu_{\eta,p})$.
  Then, there exists a constant $c>0$ such that
  \begin{equation}
    \label{eq:main}
    \sup_{f \in \Lip_{1}(\ws_{\eta, p})} \esp{f(\tS^{m})} - \esp{f(B)} 
    \le c\, \|X\|_{L^p(\Omega)}^p   \ m^{-\frac{1}{6} +\frac{\eta}{3}}\,\ln (m) .
  \end{equation}
\end{theorem}

As an application of the previous considerations, we obtain as a corollary an
approximation theorem for the local time of the Brownian motion.

The reflected Brownian motion is defined  as 
\begin{align*}
R_t = B_t + \sup_{ 0\le  s \le  t} \max \left(0,  -B_s\right)
\end{align*}
and the reflected linear interpolation of random walk is
\begin{align*}
R^m_t = X^m_t + \sup_{ 0\le  s \le  t} \max \left(0,  -X_s^m\right):=X^m_t+L_{0}^{m}(t).
\end{align*}
The process $L_{0}(t):=\sup_{ 0\le  s \le  t} \max \left(0,  -B_s\right)$ is an expression
of the local time of the Brownian motion at~$0$.   
Note that the map $f  \mapsto \left( t \mapsto f(t) + \sup_{ 0\le  s \le  t}
  \max \left(0,  -f(s)\right)\right)$ is  Lipschitz continuous from any
$W_{\eta,p}$ into $W_{0,\infty}$. One of the interest of our new result is that
we can then apply  the previous theorem in $W_{0,\infty}$ to $L_{0}^{m}$ and $L_{0}$.
We get
\begin{corollary}\label{cor-theo-princ}
Assume that the hypothesis of Theorem~\ref{thm:main} hold. There exists a
constant $c>0$ such that 
\begin{equation*}
\sup_{f \in \Lip_{1}(\ws_{0,\infty})} \esp{f(L^{m}_{0})} - \esp{f(L_{0})} \le  c \|X\|_{L^3}\, m^{-\frac{1}{6} }\ln (m) .
\end{equation*}
\end{corollary}

\section{Proofs}
\label{sec:proofs}
In what follows, $c $ denote a non significant constant which may vary from line
to line. We borrow from the current usage in rough path theory the notation
\begin{equation}\label{eq_core:22}
  f_{s,t} = f(t)-f(s).
\end{equation}

As a preparation to the proof of Theorem~\ref{propA1}, we need the following lemma.
\begin{lemma}\label{lem-mart-dis}
For all $p\ge 2 $,  there exists a constant $c_p$ such  that for any sequence of
independent, identically distributed  random variables   $(X_i, {i \in \N})$ with  $X\in L^{p}$ and any sequence  $(\alpha_i,\, i\in \N)$.
\begin{equation*}
\esp{\left| \sum_{i=1}^n \alpha_i X_i \right|^p } \le  c_p \bigl|\{i\le  n, \alpha_i\neq 0\}\,\bigr|^{p/2} (\sum_{i \le  n}|\alpha_i|^p) \ {\mathbb E}(|X|^p),
\end{equation*}
where $|A|$ is the cardinality of the set $A$.
\end{lemma}
\begin{proof}
The Burkholder-Davis Gundy inequality applied to the discrete martingale
$(\sum_{i=1}^n \alpha_i X_i, \, n\ge 0)$ yields
\begin{equation*}
\esp{ \left| \sum_{i=1}^n \alpha_i X_i \right|^p } \le  c_p \esp{ \left| \sum_{i=1}^n \alpha_i^2 X_i^2 \right|^{p/2}}.
\end{equation*}
Using Jensen inequality we obtain
\begin{equation*}
\esp{ \left| \sum_{i=1}^n \alpha_i X_i \right|^p} \le  c_p \bigl|\{i\le  n, \alpha_i\neq 0\}\bigr|^{p/2-1}\ \esp{ \sum_{i=1}^n |\alpha_i|^p X_i^p }.
\end{equation*}
The proof is thus complete.
\end{proof}
\begin{proof}[Proof of Theorem~\ref{propA1}]
 Actually, we already proved in  \cite{coutin_convergence_2017} that
\begin{equation}\label{lem3.1-la}
 \esp{ \|\tS^{m}_{s,t}\|^{p}}\le c \|X\|_{L^{p}}\ \left(\sqrt{t-s}\wedge m^{-1/2}\right).
\end{equation}
Assume that $s$ and $t$ belongs to the same sub-interval: 
There exists $l\in \{1,...,N \}$ such that $$\frac{l-1}{N} \le  s<t \le  \frac{l}{N}\cdotp$$
Then  we have (see \eqref{eq_core:22})
\begin{equation*}
  \pi^N(\tS^m)_{s,t}= \sqrt{N}\, \left(\sum_{k=1}^m X_k\,
    (h_k^m, h^N_l)_{\cm}\right) \,(t-s). 
\end{equation*}
Using Lemma \ref{lem-mart-dis}, there exists a constant $c$ such that
\begin{equation*}
\frac{\| \pi^N(\tS^m)_{s,t} \|_{L^{p}} }{\sqrt{N}\,|t-s| }
\le  c \, \|X\|_{L^{p}}\ \bigl|\{k,\,(h_k^m,h_l^N)_{\cm}\neq 0 \}\bigr|^{1/2}   \sup_k \left|(h_k^m, h^N_l)_{\cm}\right|.
\end{equation*}
Note that $|(h_k^m,h_l^N)_{\cm}|\le  \sqrt{\frac{N}{m}}$ and 
there is at most $\frac{m}{N} +2$ terms such that $ (h_k^m, h_N^l)_{\cm}$ is
non zero. Thus,
\begin{equation*}
  \label{eq:lem3.1-1}
  \frac{\| \pi^N(\tS^m)_{s,t} \|_{L^{p}} }{\sqrt{N}\,|t-s|} \le c\,  \|X\|_{L^{p}}\ \left(\frac{m}{N} +2\right)^{1/2}\sqrt{\frac{N}{m}}  \le c\,  \|X\|_{L^{p}},
\end{equation*}
as $m/N$ tends to infinity. Since $|t-s|\le 1/N$,
\begin{equation}\label{eq_core:11}
  \| \pi^N(\tS^m)_{s,t} \|_{L^{p}} \le c\,  \|X\|_{L^{p}}\, \sqrt{|t-s|}.
\end{equation}
For $0 \le  s \le  t \le  1$ let $s_+^N:=\min\{l,\,s \le  \frac{l}{N}\}$ and
$t_-^N:=\sup\{l,\,t \ge \frac{l}{N}\}$. 
We have
\begin{multline*}
  \pi^N(\tS^m)_{s,t} - \tS^m_{s,t} = \bigl(\pi^N(\tS^m)_
  {s,s_+^N} - \tS^m_{s,s_+^N}\bigr) \\
  +\bigl(\pi^N(\tS^m)_{s_+^N,t_-^N } - \tS^m_{s_+^N,t_-^N }  \bigr) + \bigl (\pi^N(\tS^m)_{t_-^N,t} - \tS^m_{t_-^N,t}\bigr).
\end{multline*}
Note that for all $f\in \ws_{\eta,p},$ $\pi^{N}(f)$ is the linear interpolation of $f$ along the subdivision $\dyad_{N}$;
hence,
for $s,t \in \dyad_{N}$, $\pi^N(\tS^m)_{s,t}=\tS^m_{s,t}$ . Thus the median term
vanishes and  we obtain
\begin{multline}\label{eq_core:14}
\esp{   \|\pi^N(\tS^m)_{s,t} - \tS^m_{s,t} \|^{p}}\le c\Bigl( \esp{\|\pi^N(\tS^m)_{s,s_+^N}\|^{p
  }} +  \esp{\|\tS^m_{s,s_+^N}\|^{p  }}\\ + \esp{\|\pi^{N}(\tS^m)_{t_-^N,t}\|^{p  }}+ \esp{\|\tS^m_{t_-^N,t}\|^{p  }} \Bigr).
\end{multline}
From \eqref{eq_core:11}, we deduce that
\begin{equation}
\label{eq_core:12}  \esp{\|\pi^N(\tS^m)_{s,s_+^N}\|^{p}}^{1/p}\le c \, \|X\|_{L^{p}}\  \sqrt{s_{+}^{N}-s}\le c \,\|X\|_{L^{p}}\  N^{-1/2},
\end{equation}
and the same holds for $ \esp{\|\pi^N(\tS^m)_{t_-^N,t}\|^{p}}$. We infer from
 \eqref{lem3.1-la}, \eqref{eq_core:11} and \eqref{eq_core:12} that
\begin{equation}
  \label{eq:continuityModulus}
   \esp{\|\pi^N(\tS^m)_{s,t}-\tS^{m}_{s,t}\|^{p}}^{1/p}\le c \, \|X\|_{L^{p}}\left(\sqrt{t-s}\wedge N^{-1/2}\right).
\end{equation}
A straightforward computation shows that
\begin{equation}\label{eq_core:13}
  \iint_{[0,1]^2} \frac {[| t-s| \wedge {N^{-1}}]^{p/2}}{|t-s|^{ 1+\eta p}} \dif s\dif t\le  c\,  N^{-p ( 1/2 -\eta)}.
\end{equation}
The result follows \eqref{eq:continuityModulus} and \eqref{eq_core:13}.
\end{proof}
\subsection{Stein method}
\label{stein}
We wish to estimate
\begin{align*}
\esp{ f( \pi^N(\tS^m) )} - \esp{ f( \pi^N(B^m) )},
\end{align*}
using the Stein's method.  For the sake of simplicity, we set
\begin{equation*}
  f_{N}=f\circ \pi^{N}.
\end{equation*}

The Stein-Dirichlet representation formula \cite{DecreusefondSteinDirichletMalliavinmethod2015}
stands that, for any ${\tau_0}>0$,
\begin{multline*}\label{eq_donsker_wiener:2}
 \esp{f_N (\wB^m)}-\esp{f_N (\tS^m)}=\esp{ \int_0^\infty
  \frac{d}{dt}P_\tau f_N (\tS^m)\dif \tau }\\
=\esp{P_{\tau_0} f_N (\tS^m)-f_N (\tS^m)}+\esp{\int_{\tau_0}^\infty
  LP_\tau f_N (\tS^m)\dif \tau},
\end{multline*}
where 
\begin{equation*}
LP_\tau f_{N}(S^{m})=-\< \tS^{m}, \nabla P_\tau f_{N}(\tS^{m }) \>_{H}+\sum_{a\in \A^{m}}\<\nabla^{(2)} P_\tau f_{N}(\tS^{m}),\, h_{a}^{m}\otimes h_{a}^{m}\>_{\cm^{\otimes 2}}.
\end{equation*}
It is straightforward (see \cite[Lemma 4.1]{coutin_convergence_2017}):
\begin{lemma}\label{lem:voisDeZero}
For any $(\eta,p)\in \Lambda$, there exists a constant $c>0$ such that for any sequence of independent,
centered random vectors  $(X_a, \, a\in \A^{m})$ such that $\esp{ \|X\|^p}
<\infty$, for 
any $f\in \Lip_{1}(W_{\eta,p})$, we have
\begin{align*}
\esp{f(\tS^m)}- \esp{ P_{\tau_0} f(\tS^m)} \le  c\, \|X\|_{L^{p}}\ \sqrt{1- e^{\tau_0}}.
\end{align*}
\end{lemma}
We now show, that as usual, the rate of convergence in the Stein's method is
related to the Lipschitz modulus of the second order derivative of the
solution of the Stein's equation. Namely, we have
\begin{lemma}
\label{lem:reductionToContinuityModulus}
For any $f\in \Lip_{1}(\ws_{\eta,p})$, we have
  \begin{multline*}
  \esp{ L P_\tau f_N (\tS^m)}\\
 \shoveleft{ =-\ \esp{\sum_{a\in \cA^m} \<\nabla^{(2)} P_\tau
                                         f_N(\tS^m_{\neg a})-\nabla^{(2)} P_\tau
                                         f_N(\tS^m),\, h^m_a\otimes h^m_a\>_{\cm^{\otimes 2}}}}\\
                                       + \ \esp{\sum_{a\in \cA}
                                         \,X_a^2\int_0^1 \ \<\nabla^{(2)} P_\tau
                                        f_{N}(\tS^m_{\neg a}+r X_a h^m_a)-\nabla^{(2)}P_\tau f_N(\tS^m_{\neg
                                           a}),\,
                                         h^m_a{}^{\otimes 2}\>_{\cm^{\otimes 2}}\kern-5pt\dif
                                         r}.
\end{multline*}
\end{lemma}
\begin{proof}[Proof of Lemma~\ref{lem:reductionToContinuityModulus}]
Let $\tS^m_{\neg a} = \tS^m- X_a h^m_a$. Since the $X_a$'s are independent,
\begin{multline*}
  \esp{ \<  \nabla P_\tau f_{N},\, \tS^m\>_{\cm}}\\
 \begin{aligned}
   &= \ \esp{\sum_{a\in \cA^m} X_a \ \left\langle {\nabla P_\tau f_N (\tS^m),h^m_a}\right \rangle_{\cm}}\\
&=\  \esp{\sum_{a\in \cA^m} X_a \ \left\langle \nabla P_\tau f_N (\tS^m)-\nabla
    P_\tau f_N (\tS^m_{\neg a}),\, h^m_a \right\rangle_{\cm}}\\
&=\esp{\sum_{a\in \cA^m} X_a^2\ \<\nabla^{(2)} P_\tau
                                         f_N (\tS^m_{\neg a}),\, h^m_a\otimes h^m_a\>_{\cm^{\otimes 2}}}\\
&  +                                      \  \esp{\sum_{a\in \cA}
                                         X_a^2\int_0^1 \
                                         \< \nabla^{(2)}P_\tau f_N (\tS^m_{\neg
                                           a}+r\, X_a h^m_a) -\nabla^{(2)}P_\tau f_N (\tS^m_{\neg
                                           a}) , h^m_a{}^{\otimes 2}\>_{\cm^{\otimes 2}}\kern-5pt\dif
                                         r},
                                       \end{aligned}
\end{multline*}
according to the Taylor formula. Since  $\esp{X_a^2}=1$, we have
\begin{multline*}
  \esp{\sum_{a\in \cA^m} X_a^2\ \<\nabla^{(2)} P_\tau
                                         f_N (\tS^m_{\neg a}),\, h^m_a\otimes h^m_a\>_{\cm^{\otimes 2}}}\\=
 \esp{\sum_{a\in \cA^m} \<\nabla^{(2)} P_\tau
                                         f_N (\tS^m_{\neg a}),\, h^m_a\otimes h^m_a\>_{\cm^{\otimes 2}}}.
\end{multline*}
The result follows by difference.
\end{proof}
The main difficulty and then the main contribution of this paper is to find an
estimate of
\begin{equation*}
\sup_{v\in \V^{m}} \<\nabla^{(2)} P_\tau f_N (v)-\nabla^{(2)} P_\tau f_N (v+ \varepsilon h_{a}^{m}),\, h^m_a\otimes h^m_a\>_{\cm^{\otimes 2}},
\end{equation*}
for any $\varepsilon.$
\begin{theorem}
  \label{thm:majoModulusOfContinuity}
 There exists a constant $c$ such that for any $\tau >0$, for any $v\in \V^{m}$, for any $f\in \Lip(\fs_{\eta,p})$,
\begin{multline}
  \label{eq:_donsker:2-l}
  \left| \left\langle\nabla^{(2)}P^{m}_\tau
     f_N (v+ \varepsilon h^m_a)-\nabla^{(2)}P^{m}_\tau
     f_N (v),\,h^m_a\otimes h^m_a\right\rangle_{\cm^{\otimes 2}}\right|\\
   \le
	c \, \frac{e^{-5\tau/2}}{\beta_{\tau/2}^2}\, \varepsilon N^{\eta-\frac{1}{2}}\sqrt{\frac{N^3}{m^3}}
\cdotp
\end{multline}
\end{theorem}
\begin{proof}[Proof of Theorem~\ref{thm:majoModulusOfContinuity}]
  We know from \cite{shih_steins_2011,coutin_convergence_2017} that we have the
  following representation: for any $h\in \cm$,
 \begin{equation}\label{eq:1}
      \left\langle\nabla^{(2)}P^{m}_\tau
    f(v),\,h\otimes
    h\right\rangle_{\cm^{\otimes 2}}\\=\frac{e^{-3\tau/2}}{\beta_{\tau/2}^2}\,
 \esp{ f\Bigl(w_{\tau}(v,B^{m},\bB^{m})\Bigr) \div h(B)\div h(\bB)}
\end{equation}
where
 \begin{equation*}
    w_{\tau}(v,y,z)=e^{-\tau/2}(e^{-\tau/2}v+\beta_{\tau/2}y)+\beta_{\tau/2}z
  \end{equation*}
and $\bB$ is an independent copy of $B$. Since the map $v$ is linear with
respect to its three arguments, 
\begin{equation*}
  f_{N}\Bigl(w_{\tau}(v,\,B^{m},\,\bB^{m})\Bigr)=f_{N}\Bigl(w_{\tau}(\pi^{N}v,\,\pi^{N}B^{m},\,\pi^{N}\bB^{m})\Bigr).
\end{equation*}
Hence,
\begin{multline}\label{eq_core:15}
\left( \frac{e^{-3\tau/2}}{\beta_{\tau/2}^2} \right)^{-1}   \left\langle\nabla^{(2)}P^{m}_\tau
    f_{N}(v),\,h\otimes
    h\right\rangle_{\cm^{\otimes 2}}\\= \esp{ f_{N}\Bigl(w_{\tau}(\pi^{N}v,\pi^{N}B,\pi^{N}\bB)\Bigr) \esp{\div h(B)\,|\, \pi^{N}B^{m}}\esp{\div h(\bB)\,|\,\pi^{N}\bB^{m}}}
\end{multline}
From Lemma \ref{lem-maj-var-esp}, we know that 
\begin{equation}
  \label{eq:varianceEspCond}
  \text{Var}\left(\esp{\div h(\bB)\,|\,\pi^{N}\bB^{m}}\right)\le c\ \frac{N}{m}
\end{equation}
for $m>8\,N$, and the same holds for the other conditional expectation. Use Cauchy-Schwarz inequality in \eqref{eq_core:15} and take
\eqref{eq:varianceEspCond} into account to 
obtain
\begin{multline}\label{eq_core:9}
 \left( \frac{e^{-3\tau/2}}{\beta_{\tau/2}^2} \right)^{-1}   \left| \left\langle\nabla^{(2)}P^{m}_\tau
     f_N (v+ \varepsilon h^m_a)-\nabla^{(2)}P^{m}_\tau
     f_N (v),\,h^m_a\otimes h^m_a\right\rangle_{\cm^{\otimes 2}}\right| \\ 
 \le c \, \left(\frac{N}{m}\right)^{2}
 \left\|w_{\tau}(\pi^{N}v,\pi^{N}B^{m},\pi^{N}\bB^{m})-w_{\tau}(\pi^{N}v+\varepsilon
   \pi^{N}h_{a}^{m},\pi^{N}B^{m},\pi^{N}\bB^{m})\right\|_{\ws_{\eta,p}}\\
 = ce^{-\tau} \varepsilon \, \left(\frac{N}{m}\right)^{2} \left\|\pi^{N}h_{a}^{m}\right\|_{\ws_{\eta,p}}
\end{multline}
since $f_{N}$ belongs to $\Lip_{1}(\ws_{\eta,p})$.
Furthermore,
\begin{equation*}
  \pi^N( h_a^m)= \sum_{b \in {\mathcal A}^N}  \<h_a^m ,h_b^N\>_{\cm} \,h_b^N.
\end{equation*}
We already know that 
\begin{equation*}
  0 \le  \<h_a^m,h_b^N\>_{\cm}\le  \sqrt{\frac{N}{m}}
\end{equation*}
and that at most two terms $\<h_a^m,h_b^N\>_{\cm}$ are non zero. Moreover, according to Lemma~\ref{lem:normeHDansDual}
\begin{equation*}
  \|h_b^N\|_{\fs_{\eta,p}}\le  c\, N^{\eta-\frac{1}{2}}.
\end{equation*}
Thus,
\begin{equation}\label{majnormepiNh}
\| \pi^N( h_a^m) \|_{\ws_{\eta,p}}\le  c \,\sqrt{\frac{N}{m}}\ N^{\eta-\frac{1}{2}}.
\end{equation}
Plug estimation \eqref{majnormepiNh}  into estimation \eqref{eq_core:9}  yields
estimate \eqref{eq:_donsker:2-l}.
\end{proof}
According to \eqref{eq:_donsker:2-l} and
Lemma~\ref{lem:reductionToContinuityModulus}, since the cardinality of $\A^{m}$
is $dm$, we obtain the following theorem.
\begin{theorem}
  \label{thm:tauAlInfini}
  If  $X_a$ belongs to $L^p$, for any ${\tau_0}>0$, there exists $c>0$ such that 
  \begin{equation}\label{eq_donsker_wiener:5}
 \esp{ \int_{\tau_0}^\infty
  L P_\tau f_{N}(\tS^m)\dif \tau } \le c\, \|X\|_{L^{p}}\ \frac{N^{1+\eta}}{\sqrt{m}}
	\, \int_{\tau_0}^\infty \frac{e^{-5\tau/2}}{1-e^{-\tau/2}}\dif \tau.
  \end{equation}
\end{theorem}
If we combine Lemma~
\ref{lem:voisDeZero} and \eqref{eq_donsker_wiener:5}, we get
\begin{multline*}
  \left| \esp{f_{N}(\tS^{m})}-\esp{f_{N}(B^{m})}  \right| \\
  \le c \|X\|_{L^{p}}\left( \sqrt{1-e^{-\tau_{0}}}+\frac{N^{1+\eta}}{\sqrt{m}}
	\, \int_{\tau_0}^\infty \frac{e^{-5\tau/2}}{1-e^{-\tau/2}}\dif \tau \right).
\end{multline*}
Optimizing with respect to $\tau_{0}$ yields Theorem~\ref{propA2}.

It remains to prove \eqref{eq:varianceEspCond}. For the sake of simplicity, we give the proof for $d=1$. The
  general situation is similar but with more involved notations.

We recall that 
\begin{equation*}
\pi^N(B^m)= \sum_{b=0}^{N-1} G_{b}^{m,N} h_{b}^N .
\end{equation*}
where 
\begin{equation}\label{def-z-appendice}
G_b^{m,N} =\sum_{ a=0}^{m-1} \<h_{a}^m,h_{b}^N\>_{\cm} \div (h_{a}^m).
\end{equation}
 
\begin{lemma}
  \label{lem:cov}
  The covariance matrix $\Gamma$ of the Gaussian vector $(G_{b}^{m,N},\,
  b=0,\cdots,N-1)$ is invertible and satisfies 
  \begin{equation}\label{eq_core:18}
    \|\Gamma^{-1}\|_{\infty }\le 2.
  \end{equation}
\end{lemma}
\begin{proof}
  Since the $h_{a}^{m}$ are orthogonal in $L^{2}$, for any $b,c\in\{0,\cdots,N-1\}, $
  \begin{equation}\label{eq_core:19}
    \Gamma_{b,c}=\sum_{a=0}^{m-1} \<h_{a}^m,h_{b}^N\>_{\cm} \<h_{a}^m,h_{c}^N\>_{\cm} .
  \end{equation}
  Since a sub-interval of $\dyad_{m}$ intersects at most two sub-intervals of
  $\dyad_{N}$, the matrix $\Gamma$ is tridiagonal.
Furthermore, we know that
  \begin{equation}\label{eq_core:20}
    0\le  \<h_{a}^m,h_{b}^N\>_{\cm} \le \sqrt{\frac{N}{m}},
  \end{equation}
  and for each $b$, there are at least $(N/m-3)$ terms of this kind which are
  equal to $(N/m)^{-1/2}$. Hence,
  \begin{equation*}
    \Gamma_{b,b}\ge (\frac{m}{N}-3)(\sqrt{\frac{N}{m}})^{2}\ge \frac{3}{4}\cdotp
  \end{equation*}
  Since $\Gamma$ is tridiagonal, this implies that it is invertible. Moreover, let $D$ be the
  diagonal matrix extracted from $\Gamma$. We have proved that
  $\|D\|_{\infty}\ge 3/4.$

  For $|b-c|=1$, there is at most one term of the sum \eqref{eq_core:19} which yields a non zero
  scalar product, hence
  \begin{equation*}
    |\Gamma_{b,c}|\le \frac{N}{m}\cdotp
  \end{equation*}
  Set  $S=\Gamma-D$. The matrix $D^{-1}S$ has at most two non null entries and
  \begin{equation*}
    \|D^{-1}S\|_{\infty}\le \frac{8}{3}\frac{N}{m}\le \frac{1}{3},
  \end{equation*}
  if $m>8N$.
By iteration, we get for any $k\ge 1$,
\begin{equation*}
  \|(D^{-1}S)^{k}\|_{\infty}\le \frac{1}{3^{k}}\cdotp
\end{equation*}
Moreover,
\begin{equation*}
  \sum_{k=0}^{\infty} (-D^{-1}S)^{k}=(\Id+D^{-1}S)^{-1}=\Gamma^{-1}D.
\end{equation*}
Thus,
\begin{equation*}
  \|\Gamma^{-1}\|_{\infty}\le  \frac{4}{3}\sum_{k=0}^{\infty}  \frac{1}{3^{k}} =2.
\end{equation*}
The proof is thus complete.
\end{proof}
\begin{lemma}\label{lem-maj-var-esp} There exists a constant $c$ which depends only on the dimension~$d$ such that for
  all $m,N$ with  $m >8N$, for any  $a \in {\mathcal A}^N$
\begin{align*}
\var{\esp{ \div (h_a^m) \/ \pi^N(B^m) }} \le  \,c\,\frac{N}{m}\cdotp
\end{align*}
\end{lemma}
\begin{proof}  
Using the  framework of Gaussian vectors, for all $a \in \{0,\cdots,m-1\}$
\begin{equation}\label{def-coeff}
\esp{ \div (h_a^m) \/ \pi^N(B^m)}= \sum_{b \in {\mathcal A}^N} C^{m,N}_{a,b} G_b^{m,N}.
\end{equation}
For any $c\in \{0,\cdots,N-1\}$, on the one hand
\begin{multline*}
  \esp{\esp{ \div (h_a^m) \/ \pi^N(B^m)}\ G_{c}}=\sum_{b=0}^{N-1} \sum_{\tau=0}^{m-1} C_{a,b}^{m,N} \, \<h_{\tau}^{m},\ h_{b}^{N}\>_{\cm}\<h_{\tau}^{m},\ h_{c}^{N}\>_{\cm}\\
  =\sum_{b=0}^{N-1} C_{a,b}^{m,N} \Gamma_{b,c}.
\end{multline*}
and on the other hand,
\begin{equation*}
   \esp{\esp{ \div (h_a^m) \/ \pi^N(B^m)}\ G_{c}}=\esp{ \div (h_a^m) \, G_{c}}=\<h_{a}^{m},\ h_{c}^{N}\>_{\cm}.
 \end{equation*}
 This means that
 \begin{equation*}
   \left( \<h_{a}^{m},\ h_{c}^{N}\>_{\cm},\, c=0,\cdots,N-1 \right) =  \left( C_{a,b}^{m,N},\, b=0,\cdots,N-1 \right)\Gamma.
 \end{equation*}
 In view of Lemma~\ref{lem:cov}, this entails that
 \begin{equation*}
   \left( C_{a,b}^{m,N},\, b=0,\cdots,N-1 \right)= \left( \<h_{a}^{m},\ h_{c}^{N}\>_{\cm},\, c=0,\cdots,N-1 \right)\Gamma^{-1}.
 \end{equation*}
 Once again we invoke \eqref{eq_core:20} and the fact that at most two of the
 terms $\<h_{a}^{m},\ h_{c}^{N}\>_{\cm}$  are non zero for a fixed $a$, to deduce
 that
 \begin{equation}\label{eq_core:21}
   \sup_{a,b}|C_{a,b}^{m,N}|\le 2 \|\Gamma^{-1}\|_{\infty} \sqrt{\frac{N}{m}}=4 \sqrt{\frac{N}{m}}\cdotp
 \end{equation}
Now then, according to  the very definition of the conditional expectation
\begin{align*}
\var{\esp{ \div (h_a^m) | \pi^N(B^m) }}&=\esp{ \div (h_a^m)\ \esp{ \div (h_a^m) | \pi^N(B^m)} }\\
&=\sum_{b=0}^{N-1} C^{m,N}_{a,b} \<h_a^m,h_{b}^{N}\>_{\cm}.
\end{align*}
Hence,
\begin{equation*}
  \var{\esp{ \div (h_a^m) | \pi^N(B^m) }} \le 2  \sup_{a,b}|C_{a,b}^{m,N}| \sqrt{\frac{N}{m}}\le 8\, \frac{N}{m},
\end{equation*}
according to \eqref{eq_core:21}.  The constant $8$ has to be modified when $d>1$.
\end{proof}

\end{document}